\documentclass[11pt]{article}

\usepackage{geometry}
\usepackage[utf8]{inputenc}

\usepackage[round]{natbib}

\usepackage{amsmath,amssymb,hyperref,amsthm,subfigure,graphicx}
\newtheorem{thm}{Theorem}[section]
\newtheorem{theorem}[thm]{Theorem}
\newtheorem{lemma}[thm]{Lemma}
\newtheorem{construction}[thm]{Construction}
\newtheorem{definition}[thm]{Definition}

\title{Bishop Independence on the Surface of a Square
Prism}
\date{June 2020}
\author{Liam H. Harris \qquad Stephanie Perkins \qquad Paul A. Roach\\~\\
Department of Computing and Mathematics\\
University of South Wales\\
Pontypridd, UK.\\
\tt\{liam.harris, stephanie.perkins, paul.roach\}@southwales.ac.uk}
\begin{document}
\maketitle
\begin{abstract}
\normalsize
Bishop Independence concerns determining the maximum number of bishops that can be placed on a board such that no bishop can attack any other bishop. This paper presents the solution to the bishop independence problem, determining the bishop independence number, for all sizes of boards on the surface of a square prism. \\

\end{abstract}

\section{Introduction}

In Chess, the independence problem for a given board asks for a maximal placement of a given chess piece such that no piece in the placement can move to the position of another in a single chess move. The independence number is the cardinality of such a placement. The relation of this problem to graph problems is of special interest to combinatorialists. For the graph $G$ with set of vertices $V$, the set $S \subset V$ is independent if no two vertices in $S$ are adjacent in $G$. The independence number $\beta_{0}(G)$ is the maximum cardinality of an independent set of the graph $G$. This directly relates to the problem of finding the independence number of a specified chess piece on a given board.\\

Results for the independence number for the bishop piece on a variety of board structures are known. \cite{Yaglom} determine the independence number of bishops for the case of $n \times n$ square boards and \cite{Berghammer} presents a proof for the independence number for bishops on the $n \times m$ rectangular board. \cite{DeMaio_Faust} consider the independence problem for bishops on the $n \times m$ torus; this paper also addresses the related domination problem which asks, for a given board, for a minimum placement of a given chess piece such that every square of the board is either occupied or can be reached in a single move by a piece in the placement. This later problem and variants have received recent interest, for example \cite{Sown_Naidu} have addressed perfect domination for bishops, kings and rooks on the $n \times n$ square board. Succinctly expressed summaries of some of these results and of other similar results, for a variety of pieces and boards, are given by \cite{Watkins}.\\

The aforementioned work only addresses board topologies created from a single $n \times m$ board, however research has also been conducted for the surface of an $n \times m \times l$ cuboid. This includes recent work by \cite{Omran} on the domination and independence of rooks and kings on the surface of the $n \times n \times n$ cube. The current authors (\cite{Harris1}) have previously presented results for the bishop independence number on the surface of the $n \times n \times n$ cube. This paper extends that work to provide the bishop independence number for the $n \times n \times m$ square prism; without loss of generality it will be assumed that $n \leq m$. We define B$^3_{n,n,m}$ to be the graph which represents permissible bishop moves on the surface
of an $n \times n \times m$ square prism.\\

The surface of a $n \times n \times m$ square prism can be considered as a net consisting of two $n \times n$ faces and four $n \times m$ faces, with the 24 edges paired and identified in a specific way such that the surface of a square prism is formed. The bishop piece moves on this board in the same manner it does on the traditional $8 \times 8$ square board moving along either of the diagonals that intersect at its current position. A bishop is said to \textit{cover} a square if it can reach that square in a single move.

\begin{definition} \label{diagonal}
A \textbf{diagonal} is a set of consecutive squares that are diagonally adjacent.
\end{definition}

Using Definition~\ref{diagonal} it is possible to define different types of diagonals, and the properties of a given diagonal, which is of use in exploring the independence problem for the bishop piece.

\begin{definition} \label{ddiagonal}
A \textbf{maximal diagonal} is a set of all consecutive squares that are diagonally adjacent. All maximal diagonals are distinct.
\end{definition}

Since a bishop can move to any square on a diagonal on which it is currently placed it is impossible for two bishops to be placed on the same diagonal and be independent. Determining the number of distinct maximal diagonals of a board can provide a useful bound to the bishop independence number of that board. To do this it will be important to distinguish between diagonals that do cross identities and diagonals that do not.

\begin{definition} \label{subdiagonal}
A \textbf{bounded diagonal} is a set of all consecutive squares that are diagonally adjacent without considering identities.
\end{definition}

When imposing the grid structure of the chessboard onto the surface of a square prism non-regularity occurs only for the 24 squares at the corners of each face. These squares have only 7 surrounding squares as opposed to the 8 neighbours possessed by the remaining squares. Thus a bishop on a corner square can only move in one of three directions rather than four. Hence, there is a need to distinguish between maximal diagonals that begin at a corner and those that do not.

\begin{definition} \label{openclosed2}
A maximal diagonal is \textbf{closed} if the squares can be ordered such that they cycle. A  maximal diagonal is \textbf{open} if it is not closed.
\end{definition}

The cuboid is an interesting object of study for Bishop movement because its closed maximal diagonals can be separated into four sets and each of its open maximal diagonals contains only squares from a single face. Determining the number of maximal diagonals and hence the independence number for bishops on a $n \times n \times m$ square prism is more complicated since the open maximal diagonals may contain squares from multiple faces and the closed maximal diagonals cannot be separated into as small a number of sets as is the case for the cube.

\begin{lemma}
\label{lemma_BishInd:Expand2}
$\beta_0(B^3_{n,n,m}) = \beta_0(B^3_{n,n,m+4n})$.
\end{lemma}
\begin{proof}
For a $4n \times 4n$ square board with top and bottom sides identified, each maximal diagonal contains a square from the first and last column. For each maximal diagonal the start and end point is the same, vertically. Hence inserting such a board into the net of an $n \times n \times m$ cuboid between two columns of the $n \times m$ faces, would extend the board to a $n \times n \times (m+4n)$ cuboid such that the maximal diagonals contain the same squares as they did in the original board. Hence the independence number of the board is not changed.
\end{proof}

\newpage

\begin{lemma}
\label{lemma_BishInd:SquareFace}
All closed diagonals of an $n \times n \times m$ cuboid pass both square $n \times n$ faces of the cuboid at least once.
\end{lemma}
\begin{proof}
Each vertical identity defined by the representation in Figures~\ref{fig_BishInd:CuboidK1A} to~\ref{fig_BishInd:CuboidK4}, labeled $1$, $2$, $3$, $4$, $5$ and $6$, of the $n \times m$ faces, identifies with an edge of a $n \times n$ face. Since movement in this representation is in one direction across the $n \times m$ faces, in order to close, a diagonal must pass one of these identities and thus pass a $n \times n$ face. Any diagonal crossing a $n \times n$ face returns to the $n \times m$ faces in the opposite orientation, and the argument applies again at least once and to the other $n \times n$ face in order for the diagonal to return to its original position in the correct orientation.
\end{proof}

\section{Bound}

\begin{lemma}
\label{lem_BishInd:SquarePrisim}
For B$^3_{n,n,kn+r}$ with $k, r \in \mathbb{Z}$ and $0\leq r<n$:
$$
\beta_0(B^3_{n,n,kn+r}) \leq
\begin{cases}
	2n + 4 & \text{for k odd and $r=0$,} \\
	2n + 3 - r & \text{for k odd and $0<r<n$,} \\
	n + 5 & \text{for k even and $r=0$,} \\
	n + 3 + r & \text{for k even and $0<r<n$.} 
\end{cases}
$$
\end{lemma}
\begin{proof}
The only boundaries to bishop movement on the surface of a cuboid are the corners of each face. Each face has 4 corners; there are 6 faces and for these 24 corners there are 12 distinct open diagonals. These are the 4 major diagonals, bounded diagonals containing the most squares, of the $n \times n$ faces and the 8 diagonals that connect the corners of the $n \times kn+r$ faces. For $r=0$, $kn+r$ is a multiple of $n$ and hence these 8 open diagonals will not pass the $n \times n$ faces. For $0< r < n$, $kn + r$ is not a multiple of $n$ and hence each of these 8 open diagonals will pass one of the $n \times n$ faces, altering the number of maximal diagonals on the cuboid compared to that of $r=0$.\\

For $0< r < n$, each $n \times n$ face is visited, exactly once, by each of four open diagonals that commence from $n \times kn+r$ faces. Selecting the leftmost $n \times n$ face, there are $n$ positive bounded diagonals incident to identity $5$ the lowest of which is in the open diagonal that is contained on this face alone. The $r^{th}$ diagonal above this is in a open diagonal which passes this square face. The remaining $n-2$ diagonals can be split into two sets, the $n-1-r$ diagonals above, and the $r-1$ diagonals below that open diagonal which passes this $n \times n$ face (see Figures~\ref{fig_BishInd:CuboidK1A} to~\ref{fig_BishInd:CuboidK4}). Taking movement across identity $7$ as granted, if necessary, the properties of all remaining diagonals of the board will be the same as for those of these two sets of diagonals. For $r=0$ the second set does not exist and the first set contains $n-1$ diagonals since there is no open diagonal that commences from a $n \times kn+r$ face which also crosses a $n \times n$ face. Hereafter it is taken that for $r=0$ there is no second set. By Lemma~\ref{lemma_BishInd:Expand2} only the cases of $k = 1,~2,~3$ and $~4$ need to be considered, since the addition of $4n$ to $m$ does not alter the independence number. \\

For $k=1$, the first diagonal set, containing $n-1-r$ diagonals for $0<r<n$ and $n-1$ diagonals for $r=0$, consists of diagonals that are closed, pass each $n \times n$ face once, and have the following path of movement: identities 1, 4 then 5. For $0<r<n$ the second diagonal set, containing $r-1$ diagonals, consists of diagonals that are closed, pass each $n \times n$ face twice, and have the following path of movement: identities 1, 4, 6, 3, 2 then 5. For $k=3$, the first diagonal set contains diagonals that are closed, pass each $n \times n$ face once, and have the following path of movement: identities 1, 6, 2 then 5. For $0<r<n$ the second diagonal set contains diagonals that are closed, pass each $n \times n$ face twice, and have the following path of movement: identities 1, 2, 3, 4, 6 then 5. In both of these cases, by symmetry there exist four sets of $n-1-r$ diagonals for $0<r<n$ (and four sets of $n-1$ diagonals for $r=0$) that traverse each square face once and for $0<r<n$ two sets of $r-1$ diagonals traversing each $n \times n$ face twice. This results in a total of $4(n-1-r) +2(r-1) + 12 = 4n -2r + 6$ maximal diagonals on the $n \times n \times kn+r$ cuboid for $k$ odd and $0<r<n$ (and a total of $4(n-1)+12=4n+8$ maximal diagonals on the $n \times n \times kn$ cuboid for $k$ odd).\\

For $k=2$, the first diagonal set, containing $n-1-r$ diagonals for $0<r<n$ and $n-1$ diagonals for $r=0$, consists of diagonals that are closed, pass each $n \times n$ face once, and have the following path of movement: identities 1, 4, 6, 3, 2 then 5. For $0<r<n$ the second diagonal set, containing $r-1$ diagonals, consists of diagonals that are closed, pass each $n \times n$ face twice, and have the following path of movement: identities 1, 6, 2 then 5. For $k=4$, the first diagonal set contains diagonals that are closed, pass each $n \times n$ face once, and have the following path of movement: identities 1, 2, 3, 4, 6 then 5. For $0<r<n$ the second diagonal set contains diagonals that are closed, pass each $n \times n$ face twice, and have the following path of movement: identities 1, 4 then 5. In both of these cases, by symmetry there exist two sets of $n-1-r$ diagonals for $0<r<n$ (and two sets of $n-1$ diagonals for $r=0$) that traverse each $n \times n$ face twice and for $0<r<n$ four sets of $r-1$ diagonals traversing each square face once. This results in a total of $2(n-1-r) +4(r-1) + 12 = 2n +2r + 6$ maximal diagonals on the $n \times n \times kn+r$ cuboid for $k$ even and $r>0$ (and a total of $2(n-1)+12=2n+10$ maximal diagonals on the $n \times n \times kn$ cuboid for $k$ even).\\

Since none of these maximal diagonals cross themselves a bishop will always cover 2 diagonals and hence the maximum number of bishops that can be placed independently on the surface of an $n \times n \times kn+r$ cuboid is $2n -r + 3$ bishops for $k$ odd and $0<r<n$, $2n + 4$ bishops for $k$ odd and $r=0$, $n+r+3$ bishops for $k$ even, and $0<r<n$ and $n+5$ bishops for $k$ even and $r=0$. 

\end{proof}

\section{Constructions}

There now follow constructions for placing a given number of bishops on an $n \times n \times m$ square prism chessboard, where $m=kn+r$ and $0\leq r<n$. 

\begin{construction}
\label{con_BishInd:nmevenA}
For $k$ odd, $n$ and $m$ even, $2n-r$ bishops are to be placed. Choosing either $n \times n$ face, from the center of either of its middle rows, for $r=0$ place $n$ bishops to the right and for $0<r<n$ place $r$ bishops in the cells to the right and a further $n-r$ to the right of these, extending onto the adjacent face if required. For the cells to the left of the center of the row place the remaining $n-r$ bishops while leaving the first $r$ cells empty, extending onto the adjacent face if required.
\end{construction}

An example of Construction~\ref{con_BishInd:nmevenA} is given in Figure~\ref{fig_BishInd:CuboidK1A}.


\begin{construction}
\label{con_BishInd:nmevenB}
For $k$ even, $n$ and $m$ even, $n+r$ bishops are to be placed. Choosing either $n \times n$ face, from the center of either of its middle rows, for $r=0$ place $n$ bishops to the right and for $0<r<n$ place $r$ bishops in the cells to the right and a further $n-r$ to the right of these, extending onto the adjacent face if required. For $0<r<n$ place the remaining $r$ bishops in the cells to the left of the center of the row.
\end{construction}

An example of Construction~\ref{con_BishInd:nmevenB} is given in Figure~\ref{fig_BishInd:CuboidK1B}.


\newpage

\begin{construction}
\label{con_BishInd:nevenmkodd}
For $k$ odd, $n$ even and $m$ odd, $2n+3-r$ bishops are to be placed, where $r$ is odd. Choosing any $n \times m$ face, beginning from the center of its $\frac{n+r+1}{2}^{th}$ column, place $\frac{2n+3-r}{2}-2$ bishops in the cells above, extending onto the adjacent rectangular face if required. Repeat this for the cells below the center of the column to place a further $\frac{2n+3-r}{2}-2$ bishops. In each of the $n \times n$ faces adjacent to the chosen rectangular face, the open diagonals intersect in four cells, arranged in two columns and two rows. In each such face, place two bishops in those cells in the column closest to the chosen $n \times m$ face.
\end{construction}

An example of Construction~\ref{con_BishInd:nevenmkodd} is given in Figure~\ref{fig_BishInd:CuboidK2A}.


\begin{construction}
\label{con_BishInd:nkevenmodd}
For $k$ even, $n$ even and $m$ odd, $n+3+r$ bishops are to be placed, where $r$ is odd. Choosing any $n \times m$ face, beginning from the center of its $\frac{2n+r+1}{2}^{th}$ column, place $\frac{n+3+r}{2}-r-1$ bishops in the cells above, extending onto the adjacent rectangular face if required. Repeat this for the cells below the center of the column to place a further $\frac{n+3+r}{2}-r-1$ bishops. In each of the $n \times n$ faces adjacent to the chosen $n \times m$ face, the open diagonals intersect in four cells, arranged in two columns and two rows. In each such face, place two bishops in those cells in the column closest to the chosen $n \times m$ face. Then in one of the $n \times n$ faces, a further $2r-2$ bishops are to be placed, $r-1$ of which will be placed between the bishops already placed; the remaining $r-1$ bishops are placed in the same rows as the latter $r-1$ bishops but in the other column in which the open diagonals intersect.
\end{construction}

An example of Construction~\ref{con_BishInd:nkevenmodd} is given in Figure~\ref{fig_BishInd:CuboidK2B}.


\begin{construction}
\label{con_BishInd:nkodd}
For $k$ odd, $n$ odd, $2n+3-r$ bishops are to be placed for $0<r<n$ and $2n+4$ bishops are to be placed for $r=0$. Choose an $n \times m$ face and denote its middle row as the pivotal row; for $0<r<n$ place $r-1$ bishops in this $n \times m$ face symmetrically around the center of the pivotal row. In one of the $n \times n$ faces, a bishop is placed in the center cell; two further bishops are placed in the middle row, coinciding with the pivotal row, in the cells at which the open diagonals intersect, extending into the adjacent rectangular faces if necessary. For this pair of placements, $n-r-1$ bishops for $0<r<n$ and $n-1$ bishops for $r=0$ are placed in the cells left of the left most placed bishop. Similarly $n-r-1$ bishops for $0<r<n$ and $n-1$ bishops for $r=0$ are placed to the right of the right most placed bishop. For the second $n \times n$ face, a bishop is placed in the center cell; two further bishops are placed in the middle column, perpendicular to the pivotal row, in the cells at which the open diagonals intersect, extending into the adjacent rectangular faces if necessary.
\end{construction}

An example of Construction~\ref{con_BishInd:nkodd} is given in Figure~\ref{fig_BishInd:CuboidK3}.


\begin{construction}
\label{con_BishInd:noddkeven}
For $k$ even, $n$ odd, $n+3+r$ bishops are to be placed for $0<r<n$ and $n+5$ bishops are to be placed for $r=0$. Choose an $n \times m$ face and denote its middle row as the pivotal row; $n-r-1$ bishops for $0<r<n$ and $n-1$ bishops for $r=0$ are placed in this $n \times m$ face symmetrically around the center of the pivotal row. In each of the $n \times n$ faces, a bishop is placed in the center cell; two further bishops are placed in the middle row, coinciding with the pivotal row, in the cells at which the open diagonals intersect, extending into the adjacent rectangular faces if necessary. For $0<r<n$, for one such pair of placements at the intersection of the open diagonals a bishop is placed in each empty cell between them, accounting for the remaining $2r-2$ bishops. 
\end{construction}

An example of Construction~\ref{con_BishInd:noddkeven} is given in Figure~\ref{fig_BishInd:CuboidK4}.


\section{Result}

\begin{theorem}
\label{thm_BishInd:SquarePrisim}
For B$^3_{n,n,kn+r}$ with $k, r \in \mathbb{Z}$ and $0\leq r<n$:
$$
\beta_0(B^3_{n,n,kn+r}) =
\begin{cases}
	2n + 4 & \text{for $k$ odd, $n$ odd and $r=0$,} \\
	2n + 3 - r & \text{for $k$ odd and $n$ odd and $r>0$ or for $k$ odd and $r$ odd,} \\
	2n - r & \text{for $k$ odd and $n$ and $r$ even,} \\
	n + 5 & \text{for $k$ even, $n$ odd and $r=0$,} \\
	n + 3 + r & \text{for $k$ even, $n$ odd and $r>0$ or for $k$ even and $r$ odd,} \\
	n + r & \text{for $k$ even and $n$ and $r$ even.}
\end{cases}
$$
\end{theorem}
\begin{proof}

For $n$ odd, the bound given in Lemma~\ref{lem_BishInd:SquarePrisim} can be met. For $k$ even, Construction~\ref{con_BishInd:noddkeven} provides an independent placement of $n+3+r$ bishops for $0<r<n$ and $n+5$ bishops for $r=0$. For $k$ odd, Construction~\ref{con_BishInd:nkodd} provides an independent placement of $2n+3-r$ bishops for $0<r<n$ and $2n+4$ bishops for $r=0$.
\newline
\\
For $n$ even and $nk+r$ odd, then $r\neq0$ and the bound given in Lemma~\ref{lem_BishInd:SquarePrisim} can be met. For $k$ and $r$ odd, Construction~\ref{con_BishInd:nevenmkodd} provides an independent placement of $2n+3-r$ bishops. For $k$ even and $r$ odd, Construction~\ref{con_BishInd:nkevenmodd} provides an independent placement of $n+3+r$ bishops.
\newline
\\
For $n$ and $nk+r$ even, and hence $r$ even, the bound given in Lemma~\ref{lem_BishInd:SquarePrisim} cannot be met. For $k$ odd, $2n-r$ bishops can placed independently as given in Construction~\ref{con_BishInd:nmevenA}. For $k$ even $n+r$ bishops can be placed independently as given in Construction~\ref{con_BishInd:nmevenB}. Each of these constructions provides a lower bound for the independence number of their cases. Consider the set of $n-r-1$ maximal diagonals for $0<r<n$ (and $n-1$ maximal diagonals for $r=0$) which pass the left square face via identities 1 and 5 passing each square face once for $k$ odd and twice for $k$ even, these diagonals form a cycle set; by symmetry there are four such cycle sets for $k$ odd and two for $k$ even. For $0<r<n$, consider also the set of $r-1$ maximal diagonals which pass the left square face via identities 1 and 5 passing each square face twice for $k$ odd and once for $k$ even, these diagonals also form a cycle set; by symmetry there are two such cycle sets for $k$ odd and four for $k$ even. By numbering consecutively from 1 the members of each cycle set, each set can be further split into two groups (referring to the value of each diagonal's numbering): the odd members $\mathfrak{P}$ of which there will be $\frac{n - r}{2}$ for each set of the first kind and for $0<r<n$ there will be $\frac{r}{2}$ for each set of the second kind; and the even members $\mathfrak{Q}$ of which there will be $\frac{n - r - 2}{2}$ for each set of the first kind and for $0<r<n$ there will be $\frac{r-2}{2}$ for each set of the second kind. Since $n$ and $nk+r$ are even, only the members of $\mathfrak{P}$ of each cycle set intersect the 12 open diagonals on a square. Further, the members of $\mathfrak{Q}$ only intersect members of $\mathfrak{P}$ of other cycle sets on a square. Thus, placed anywhere, a bishop covers a member of $\mathfrak{P}$, and hence an upper bound on the maximum number of bishops that can be placed independently on the board is $|\mathfrak{P}|$ which is $2n-r$ for $k$ odd and $n-r$ for $k$ even.

\end{proof}

\newpage

\nocite{*}
\bibliographystyle{abbrvnat}
\bibliography{Ind_SqPrism_Pre_Sub}
\label{sec:biblio}

\newpage

\begin{figure}[h!]
\centering
	\includegraphics[width=0.95\textwidth]{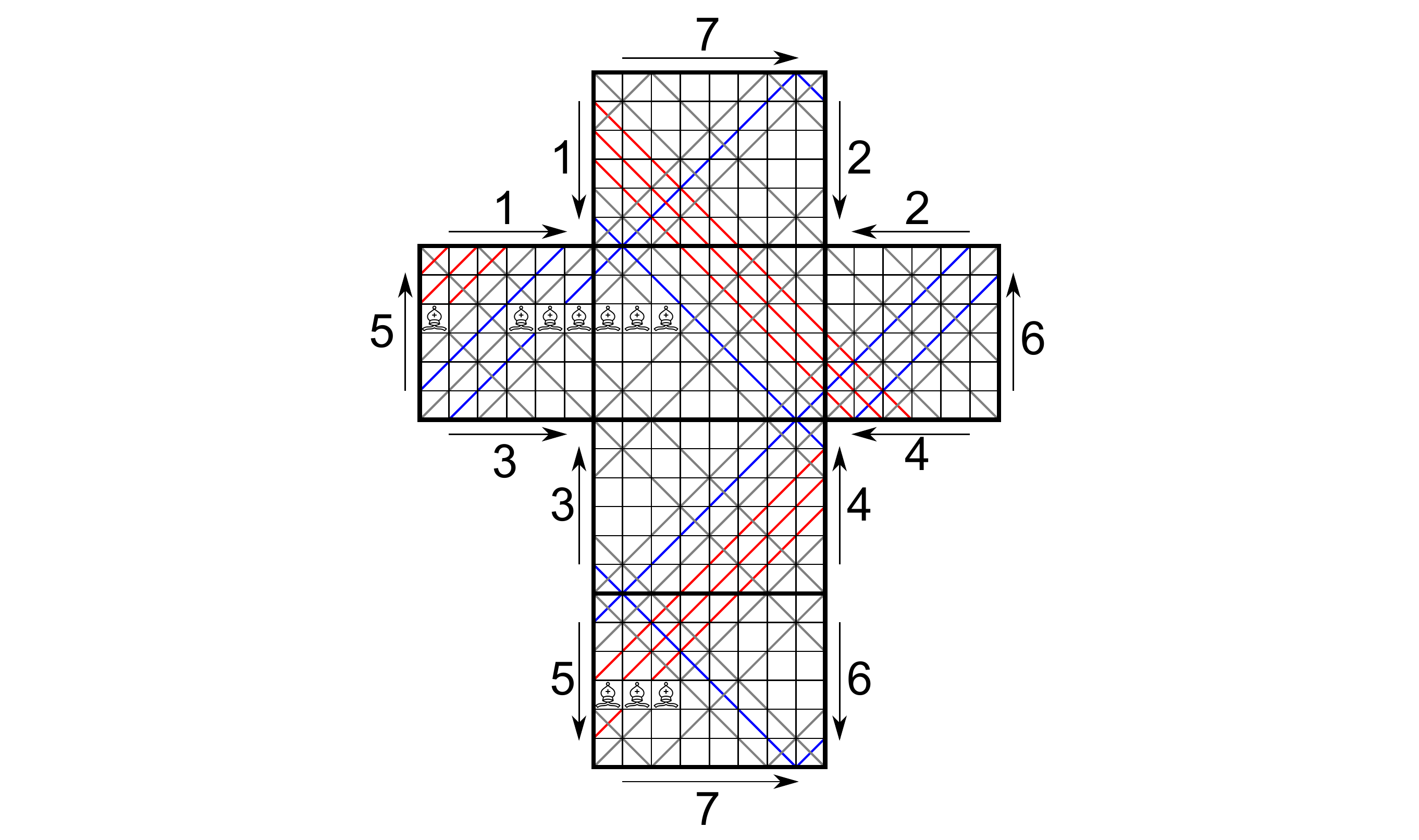}
	\caption{Independent placement of bishops for $k$ odd, $n$ and $m$ even: a 6 $\times$ 6 $\times$ 8 Cuboid.}
	\label{fig_BishInd:CuboidK1A}
\end{figure}

\begin{figure}[h!]
\centering
	\includegraphics[width=0.95\textwidth]{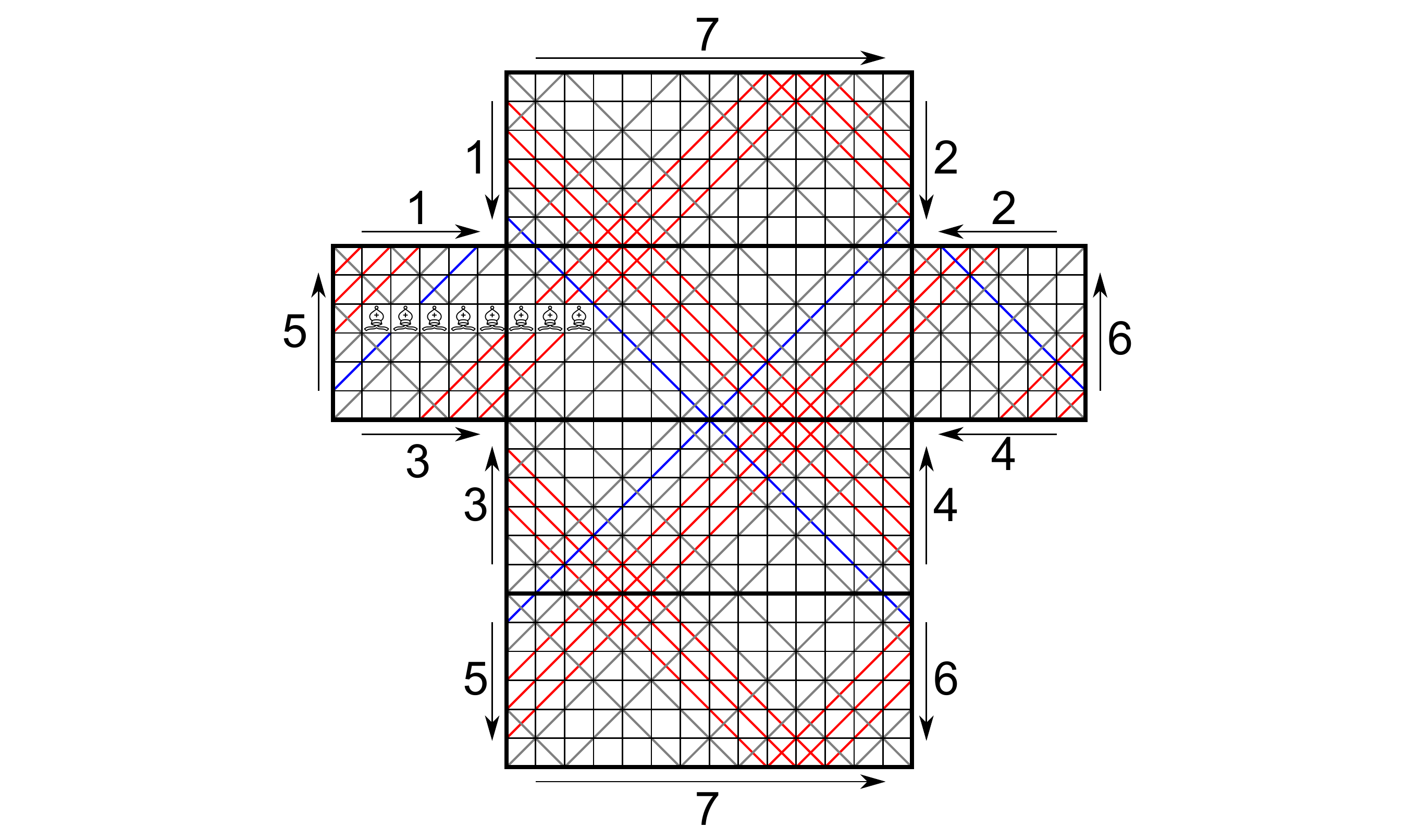}
	\caption{Independent placement of bishops for $k$, $n$ and $m$ even: a 6 $\times$ 6 $\times$ 14 Cuboid.}
	\label{fig_BishInd:CuboidK1B}
\end{figure}

\begin{figure}[h!]
\centering
	\includegraphics[width=0.95\textwidth]{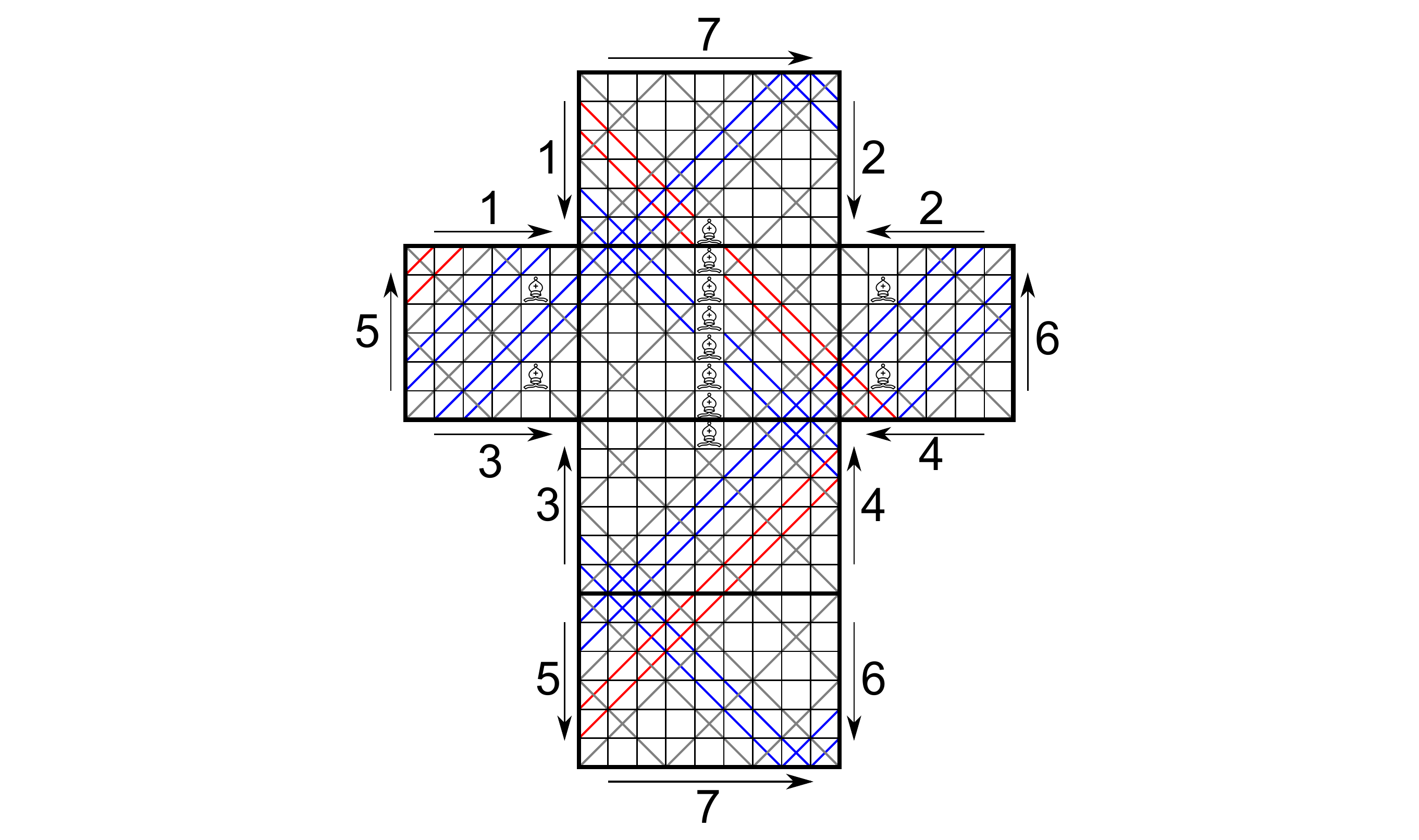}
	\caption{Independent placement of bishops for $k$ odd, $n$ even and $m$ odd: a 6 $\times$ 6 $\times$ 9 Cuboid.}
	\label{fig_BishInd:CuboidK2A}
\end{figure}

\begin{figure}[h!]
\centering
	\includegraphics[width=0.95\textwidth]{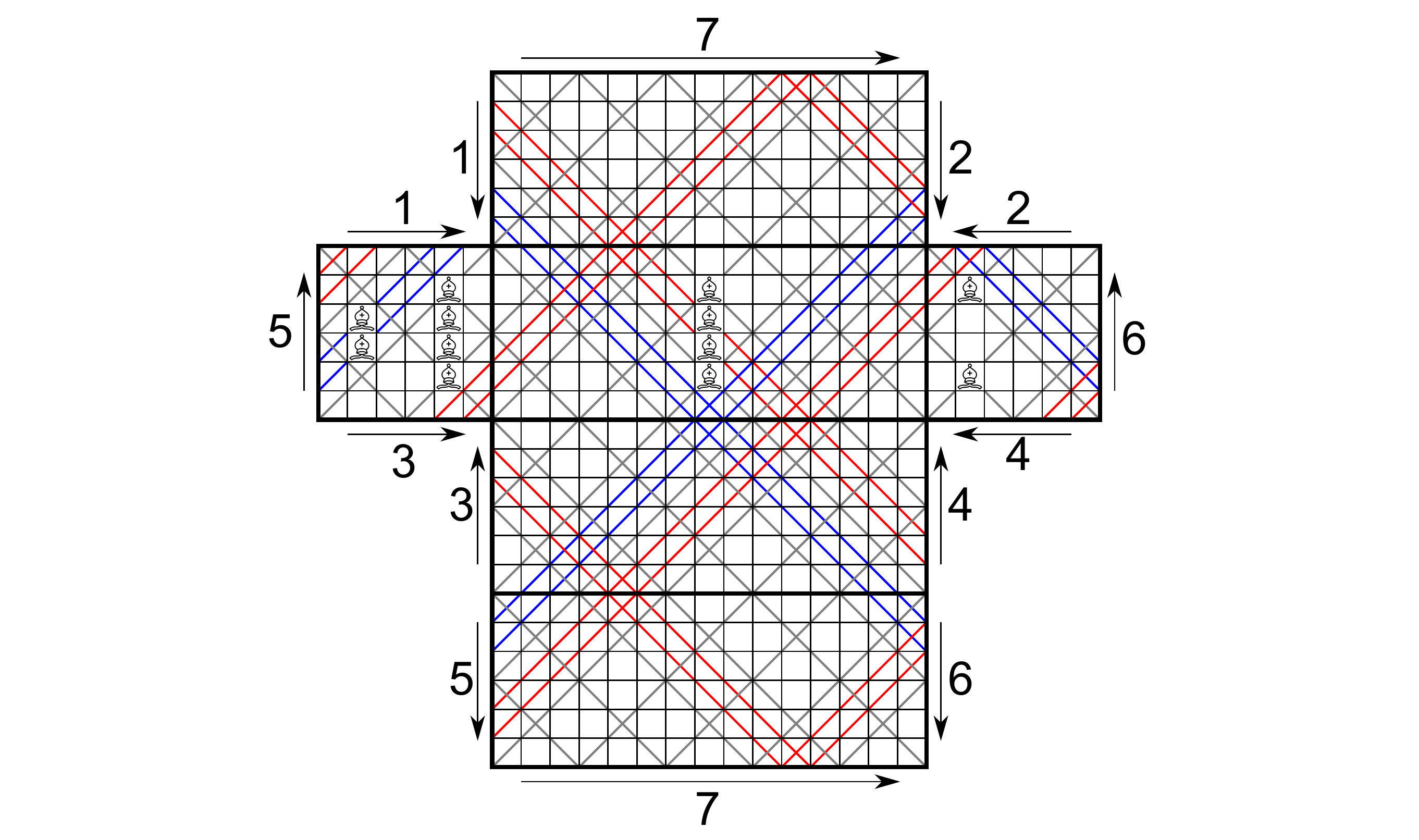}
	\caption{Independent placement of bishops for $k$ even, $n$ even and $m$ odd: a 6 $\times$ 6 $\times$ 15 Cuboid.}
	\label{fig_BishInd:CuboidK2B}
\end{figure}

\begin{figure}[h!]
\centering
	\includegraphics[width=0.95\textwidth]{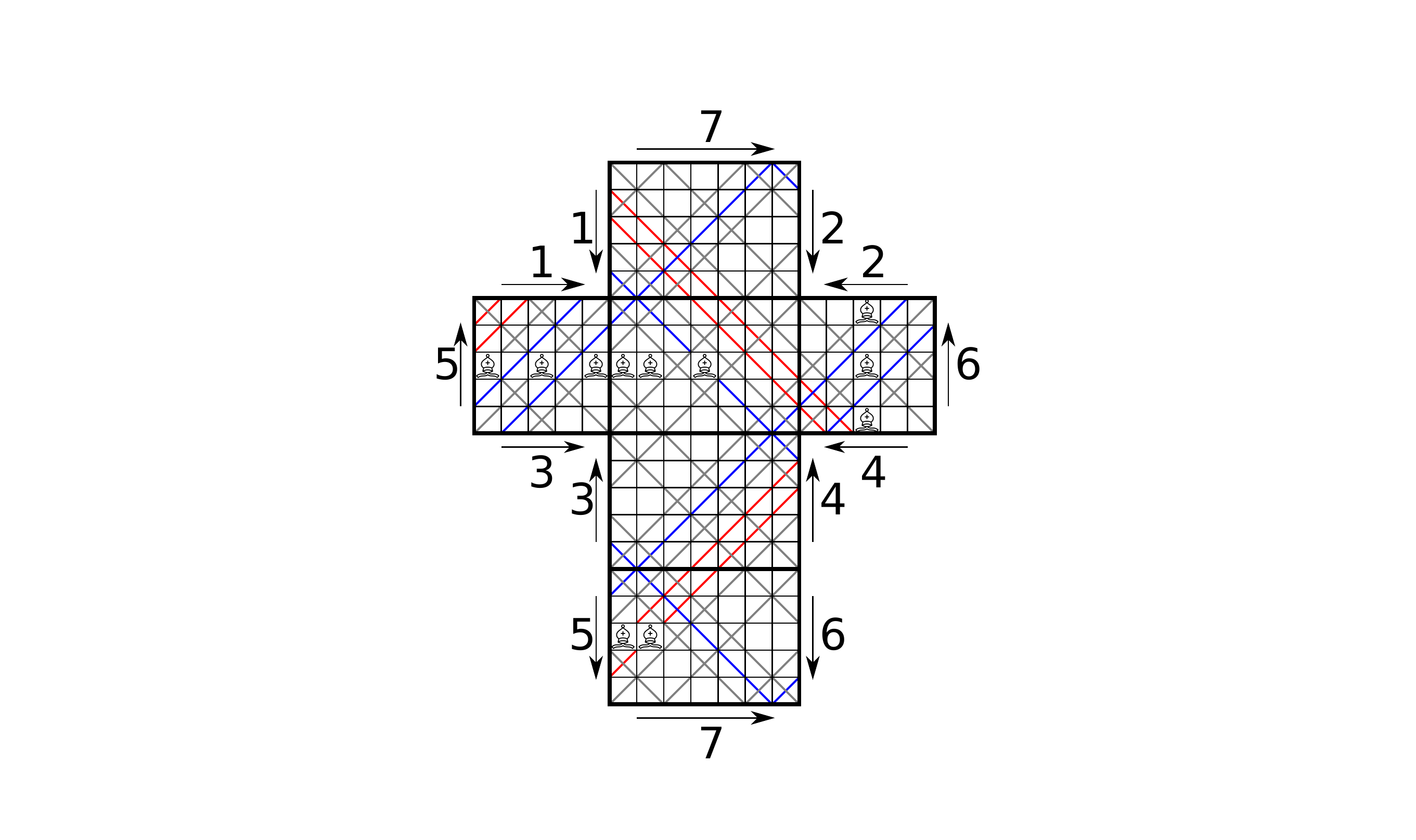}
	\caption{Independent placement of bishops for $k$ and $n$ odd: a 5 $\times$ 5 $\times$ 7 Cuboid.}
	\label{fig_BishInd:CuboidK3}
\end{figure}

\begin{figure}[h!]
\centering
	\includegraphics[width=0.95\textwidth]{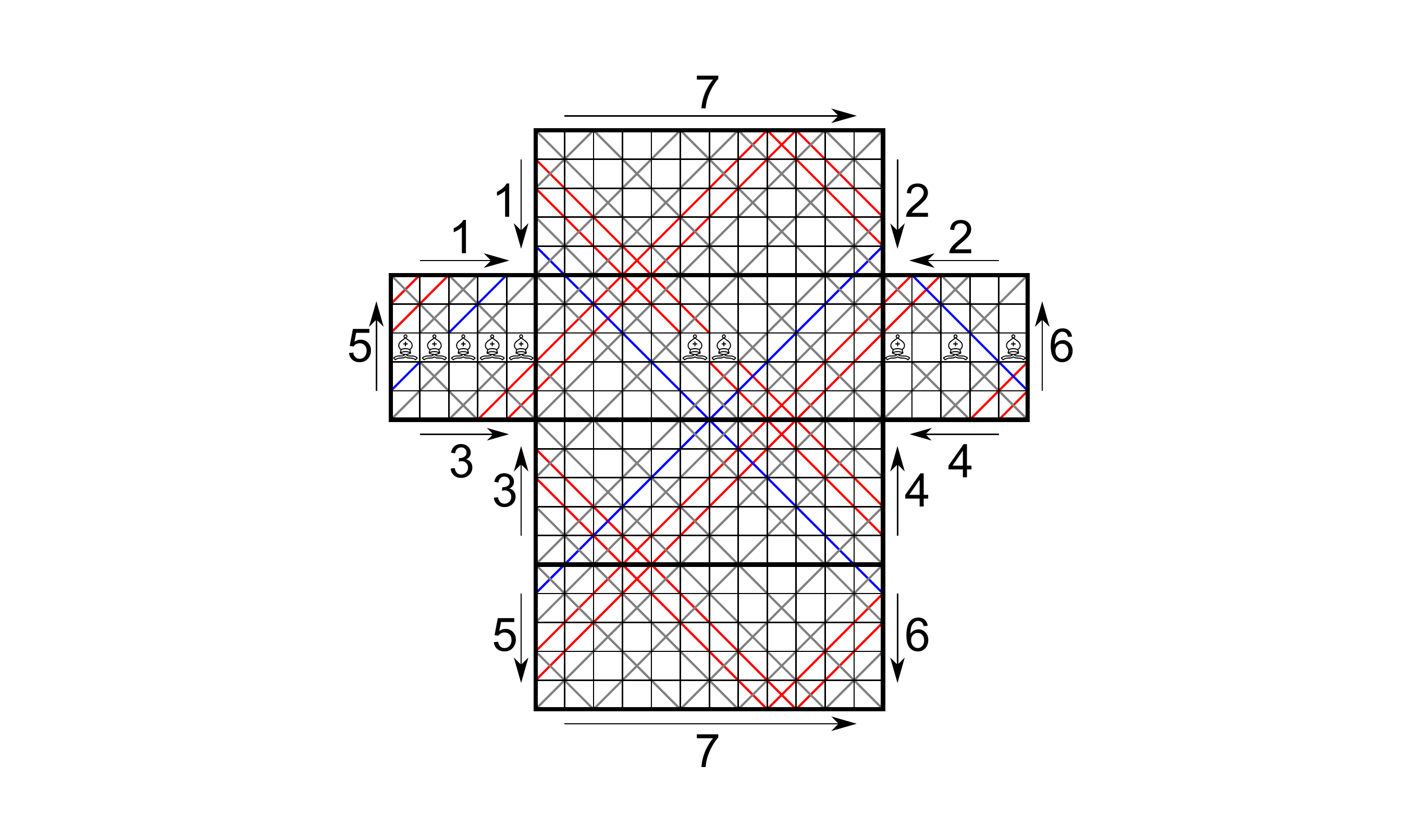}
	\caption{Independent placement of bishops for $n$ odd and $k$ even: a 5 $\times$ 5 $\times$ 12 Cuboid.}
	\label{fig_BishInd:CuboidK4}
\end{figure}

\end{document}